\documentclass[11pt]{article}
\usepackage[a4paper,margin=1in,footskip=0.25in]{geometry}
\usepackage{verbatim}
\usepackage[utf8]{inputenc}
\usepackage{amsmath,amssymb,amsthm}
\usepackage{tikz-cd}
\usepackage{pgfplots}
\pgfplotsset{compat=1.11}
\usepgfplotslibrary{fillbetween}
\usetikzlibrary{intersections}
\usepackage{bbm}
\usepackage{hyperref}
\usepackage{comment}
\usepackage{stmaryrd}
\usepackage[new]{old-arrows}
\usepackage{graphicx}
\graphicspath{ {./Images/} }
\usepackage{todonotes}

\usepackage [english]{babel}
\usepackage [autostyle, english = american]{csquotes}
\MakeOuterQuote{"}

\numberwithin{equation}{section}

\theoremstyle{definition}

\newtheorem{remark}{Remark}
\newtheorem*{example}{Example}
\newtheorem{claim}{Claim}
\newtheorem{theorem}{Theorem}
\newtheorem*{main theorem}{Main Theorem}

\newtheorem*{theorem*}{Theorem}

\pgfdeclarelayer{bg}
\pgfsetlayers{bg,main}

\begin{document}
\title{Generalized Fiber Contraction Mapping Principle}
\author{Alexandro Luna and Weiran Yang}
\date{}
\maketitle

\begin{abstract}
We prove a generalized non-stationary version of the fiber contraction mapping theorem. It was originally used in \cite{hp} to prove that the stable foliation of a $C^2$ Anosov diffeomorphism of a surface is $C^1$. Our generalized principle is used in \cite{Lu2}, where an analogous regularity result for stable foliations of non-stationary systems is proved. The result is stated in a general setting so that it may be used in future dynamical results in the random and non-stationary settings, especially for graph transform arguments. 
\end{abstract}

\section*{Introduction}

The contraction mapping principle is a classical result in mathematical analysis that has an numerous applications in the theory of iterated function systems, Newton's method, the Inverse and Implicit Function Theorems, ordinary and partial differential equations, and more (see \cite{BS}, and references therein, for a survey of applications). Many versions of this principle and converses have been studied and examined in different spaces. A detailed historical note of this theorem can be found in \cite{JJT}. 

This principle is frequently used in various areas of dynamical systems, especially in smooth dynamics. Since the early 1970s, it has been used in graph transform arguments to prove various existence and regularity results of stable foliations of hyperbolic systems \cite{hp, HPS}.

Recently, hyperbolic dynamics have been used in studying the so-called trace maps \cite{C, Ca, DG1}. Understanding the dynamical behavior of these maps is a useful tool in deriving spectral properties of discrete Schr\"odinger operators with Sturmian potential \cite{bist, D2000, degt, dg3, DGY, GJK, L, M}. One promising approach to advance these results is to further develop the theory in the non-stationary case. 

In the random and non-stationary settings, existence and smoothness of stable manifolds is well understood (see for example Chapter 7 of \cite{Ar}). In the non-stationary or non-autonomous settings, questions regarding dynamical properties of Anosov families such as existence of stable manifolds \cite{Mu}, openness in the space of two-sided sequences of diffeomorphisms \cite{Mu1}, and structural stability \cite{CRV, Mu2} have been addressed. When it comes to regularity of non-stationary stable foliations, only partial results are available, such as when the sequence of maps has a constant tail \cite{S} or for a neighborhood of a common fixed point of the maps \cite{ZLZ}, but our overall goal is to derive regularity results of these foliations, currently not available in the literature. Our primary motivation comes from questions on spectral properties of Sturmian Hamiltonians. 

This note is dedicated to providing the preliminary technical contraction mapping principles that will be useful in these non-stationary settings. 
In \cite{Lu2}, it is proved that the non-stationary stable foliation of a collection of diffeomorphisms of $\mathbb T^2$ that satisfy a common cone condition, and have uniformly bounded $C^2$ norms, is a $C^1$ foliation of $\mathbb T^2$. This result generalizes the classical version in \cite{hp} where it is proved that the stable foliation of a $C^2$ Anosov diffeomorphism of a surface is $C^1$. In \cite{hp}, a fibered version of the contraction mapping principle is used to prove this $C^1$ smoothness, and this paper is dedicated to supplying the appropriate generalized version of this principle to be applicable in \cite{Lu2} and future analogous results.  

Given complete metric spaces $X$ and $Y$, we consider a sequence of maps $(f_n)$, $f_n:X\rightarrow X$, and for each $x\in X$, a sequence of maps $\left(h_n^x\right)$, $h_n^x:Y\rightarrow Y$. Given a sequence of skew maps $(F_n)$ via 
$$F_n:X\times Y\rightarrow X\times Y, (x,y)\mapsto \left(f_n(x), h_n^x(y)\right),$$
we show that under uniform contraction rates and reasonable continuity and bounded orbit assumptions that there is a $\left(x^*, y^*\right)\in X\times Y$ such that 
$$\lim\limits_{n\rightarrow\infty} F_1\circ \cdots \circ F_n(x,y)=\left(x^*, y^*\right)$$
for all $(x,y)\in X\times Y$.
Outside of its direct application in \cite{Lu2}, this result has the potential to be used for various dynamical techniques in the random or non-stationary settings. This paper is a result of an undergraduate research project, supervised by the first author, that occurred during the Summer and Fall quarters of 2024.
\subsection*{Background and Main Results}
Given a metric space $(X,d)$ and a mapping $f:X\rightarrow X$, we define the \textit{Lipschitz constant} of $f$ to be 
$$\text{Lip}(f):=\sup_{x_1\neq x_2} \frac{d\left(f(x_1), f(x_2)\right)}{d(x_1,x_2)}.$$
If $\text{Lip}(f)<1$, then we say that $f$ is a \textit{contraction} on $X$. An element $x^*\in X$ is a \textit{fixed point} of $f$ if $$f\left(x^*\right)=x^*.$$
\begin{theorem}[Contraction Mapping Principle]
If $X$ is a complete metric space and $f:X\rightarrow X$ is a contraction on $X$, then $f$ has a unique fixed point $x^*\in X$ and moreover, 
$$\lim\limits_{n\rightarrow\infty} f^n(x)=x^*$$
for all $x\in X$.
\end{theorem}

Our goal is to generalize the following fibered version of this principle. 

\begin{theorem}[Fiber Contraction Principle \cite{hp}]\label{Fiber Contraction Theorem}
    Let $X$ be a space, $Y$ be a metric space, $f:X\rightarrow X$ a mapping, and $\{g_x\}_{x\in X}$ a family of maps $g_x:Y\rightarrow Y$ such that 
    $$F:X\times Y\rightarrow X\times Y, \ (x,y)\mapsto \left( f(x), g_x(y)\right)$$
    is continuous. Suppose that $p\in X$ is a fixed point of $f$ satisfying $\lim\limits_{n\rightarrow\infty}f^n(x)=p$
    for all $x\in X$, $q\in Y$ is a fixed point of $g_p$, and 
    $$\limsup\limits_{n\rightarrow\infty} \text{Lip} \left(g_{f^n(x)}\right)<1$$
    for all $x\in X$. Then, $(p,q)\in X\times Y$ is a fixed point of $F$ satisfying 
    $$\lim\limits_{n\rightarrow\infty}F^n(x,y)=(p,q),$$
    for all $(x,y)\in X\times Y$.
\end{theorem}
The following theorem is a non-stationary version of the Contraction Mapping Principle. 
\begin{theorem}\label{base lemma}
   Let $(X,d)$ be a complete metric space and $(f_n)$, $f_n:X\rightarrow X$, a sequence of contractions. If 
    $$\mu:=\sup_{n\in\mathbb N}\text{Lip}(f_n)<1$$
    and there is a $x_0\in X$ such that $\left(d(f_n(x_0), x_0)\right)$ is bounded, then there is a $x^*\in X$ such that 
    $$\lim\limits_{n\rightarrow\infty}f_1\circ \cdots \circ f_n(x)=x^*$$
    for all $x\in X$.
\end{theorem}

The next theorem is a non-stationary version of the Fiber Contraction Mapping Principle. 

\begin{theorem}\label{main theorem}
     Let $X$ and $Y$ be complete metric space. Suppose that $(f_n)$, $f_n:X\rightarrow X$, is a sequence of mappings such that 
     \begin{equation}\label{base contraction rate}
         \mu := \sup\limits_{n\in\mathbb N}\text{Lip}\left(f_n\right)<1.
     \end{equation}
     For each $x\in X$, let $\left(h_n^x\right)$, $h_n^x:Y\rightarrow Y$, be a sequence of mappings such that 
 \begin{equation}\label{fiber contraction rate}
     \lambda := \sup_{n\in\mathbb N}\sup_{x\in X}\text{Lip}\left(h_n^x\right)<1.
 \end{equation}
 Suppose that
    \begin{itemize}
        \item[(1)] There is a $x_0\in X$ such that $\left\{f_n(x_0)\right\}_{n\in\mathbb N}$ is bounded in $X$;
        \item[(2)] For any bounded set $\Omega\subset X\times Y$, the set $\left\{h_n^x(y)\right\}_{(x,y)\in \Omega, n\in\mathbb N}$ is bounded in $Y$;
        \item[(3)] For any bounded set $K\subset Y$ and any $n\in\mathbb N$, we have 
    $\lim\limits_{\left|x-x'\right|\rightarrow0}d_Y\left(h_n^x(y), h_n^{x'}(y)\right)=0$, and this limit is uniform  in $y\in K$.
    \end{itemize}
  Then, for the skew-maps $F_n:X\times Y\rightarrow X\times Y$, defined via $F_n(x,y):=\left(f_n(x), h_n^x(y)\right)$, there is a $\left(x^*,y^*\right)\in X\times Y$ such that 
    $$\lim\limits_{n\rightarrow \infty}F_1\circ\cdots\circ F_n(x,y)=\left(x^*,y^*\right)$$
    for all $(x,y)\in X\times Y$.
\end{theorem}

\section{Proofs of the Main Theorems}

We first prove Theorem \ref{base lemma}.
\begin{proof}[Proof of Theorem \ref{base lemma}.] Denote 
$$M:=\sup_{n\in\mathbb N}\left\{d(f_n(x_0), x_0)\right\}.$$
Suppose $n,m\in\mathbb N$ with $m>n$. Then, 
\begin{align} \label{Cauchy seq boundary}
    d\left(f_1\circ\cdots \circ f_n(x),f_1\circ\cdots \circ f_m(x)\right)\leq \mu^n d(f_{n+1} \circ \ldots \circ f_m(x_0),x_0),
    \end{align}
and by the triangle inequality, 
    \begin{align*}\notag 
        &\hspace{0.2in}d(f_{n+1} \circ \ldots \circ f_m(x_0),x_0)\\
        &\leq d\left(f_{n+1}(x_0), x_0\right)+\sum_{i=2}^{m-n} d\left(f_{n+1}\circ\cdots \circ f_{n+i}(x_0), f_{n+1}\circ\cdots \circ f_{n+i-1}(x_0)\right) \\  \notag
        &\leq d\left(f_{n+1}(x_0), x_0\right)+\sum_{i=2}^{m-n} \mu^{i-1} d\left(f_{n+i}(x_0), x_0\right)\\ \notag
        &\leq M\sum_{i=1}^{m-n}\mu^{i-1}\leq M\sum_{i=1}^{\infty} \mu^{i-1}.
    \end{align*}
    It follows that 
    $$ d\left(f_1\circ\cdots \circ f_n(x),f_1\circ\cdots \circ f_m(x)\right)\leq \mu^n M \sum_{i=1}^{\infty} \mu^{i-1}\rightarrow 0,$$
    as $n,m\rightarrow\infty$, so that $\left(f_1\circ\cdots\circ f_n(x_0)\right)$ is a Cauchy sequence. Since $X$ is complete, there is a $x^*\in X$ such that 
 $$\lim\limits_{n\rightarrow\infty}f_1\circ \cdots \circ f_n(x_0)=x^*.$$
    If $x\in X$, then 
    $$d\left(f_1\circ\cdots\circ f_n(x), f_1\circ\cdots\circ f_n(x_0)   \right)\leq \mu^nd(x,x_0)\rightarrow 0$$
    as $n\rightarrow \infty$, and hence 
     $$\lim\limits_{n\rightarrow\infty}f_1\circ \cdots \circ f_n(x)=x^*.$$
\end{proof}
\begin{remark}\label{remark on boundedness assumption in base lemma}
    Notice that from the proof, we deduced that the condition that $\left(d(f_n(x_0), x_0)\right)$ is a bounded sequence implies that the set 
    $\{f_k\circ\cdots \circ f_l( x_0)\}_{k\leq l}$ is bounded in $X$, due to the uniform contraction rates of the maps.

    We also note that this condition cannot be removed. As an example, let $X=\mathbb R$ and define $f_n:\mathbb R\rightarrow \mathbb R$ via $f_n(x):=\frac{1}{2} x+ 3^n$. Then, we have that $$f_1\circ\cdots\circ f_n(0)=\sum_{i=1}^n \frac{3^i}{2^{i-1}}\rightarrow \infty$$
    as $n\rightarrow \infty$.
\end{remark}

We now prove Theorem \ref{main theorem}.

\begin{proof}[Proof of Theorem \ref{main theorem}.]
If $\pi_Y:X\times Y\rightarrow Y$ is the projection map $\pi_Y(x,y)=y$, then for each $k\leq n$, we have 
\begin{equation}\label{projected coordinate}
    \pi_Y\circ F_k\circ \cdots \circ F_n(x,y)=h_k^{f_{k+1}\cdots f_nx}\cdots h_n^x(y).
\end{equation}
 From Remark \ref{remark on boundedness assumption in base lemma}, we know there is a $M=M(x_0)>0$ such that
            \begin{equation}\label{base iteration bound}
                d_{X}\left(f_k\circ \cdots \circ f_l (x_0), x_0\right)<M
            \end{equation}
        for all $k\leq l$. Fix $y_0\in Y$. By condition (2), there is an $S>0$ such that 
$$d_Y\left(h_n^{x}(y_0), y_0\right)<S$$
for all $n\in\mathbb N$ and $x\in B_M(x_0)$. From the triangle inequality, (\ref{fiber contraction rate}), and (\ref{base iteration bound}), if $k\leq n$, we have 
\begin{align*}
&d_Y\left(h_k^{f_{k+1}\cdots f_n x_0}\cdots h_n^{x_0}(y_0), y_0\right)\leq d_Y\left( h_k^{f_{k+1}\cdots f_n x_0}(y_0), y_0\right)  \\
&\hspace{2cm}+\sum_{i=2}^{n-k}d_{Y}\left(h_k^{f_{k+1}\cdots f_n x_0}\cdots h_{k+j}^{f_{k+j+1}\cdots f_n x_0}(y_0), h_k^{f_{k+1}\cdots f_n x_0}\cdots h_{k+j-1}^{f_{k+j}\cdots f_n x_0}(y_0)\right)\\
&\leq d_Y\left( h_k^{f_{k+1}\cdots f_n x_0}(y_0), y_0\right)+\sum_{i=2}^{n-k} \lambda^{i-1} d_Y\left( h_{k+j}^{f_{k+j+1}\cdots f_n x_0}(y_0), y_0\right)\leq S\sum_{i=1}^{n-k}\lambda ^{i-1}.  
    \end{align*}
   That is,   
\begin{equation}\label{Upper bound for iterations in Y}
            d_Y\left( h_k^{f_{k+1}\circ\cdots\circ f_n x_0}\circ\cdots\circ h_{n-1}^{f_nx_0}\circ h_n^{x_0}(y_0) ,y_0\right) < L:=S\sum_{i=1}^{\infty} \lambda^{i-1}
\end{equation}
for all $k\leq n$.

    \begin{claim}\label{existence of limit}
        There is a $y^*\in Y$ such that $$\lim\limits_{n\rightarrow\infty}h_1^{f_2\cdots f_n x_0}\circ \cdots \circ h_n^{x_0}(y_0)=y^*.$$
    \end{claim}  

    \begin{proof}[Proof of Claim \ref{existence of limit}.]
        Let $\epsilon>0$. Choose $N_0$ such that 
        \begin{equation}\label{choice of N_0 claim 1}
            2\lambda^{N_0-1}L<\epsilon.
        \end{equation}
        By condition (3), there is a $\delta>0$ such that if $d_X(x,x')<\delta$, $x,x'\in  B_M(x_0)$, then 
        \begin{equation}\label{Uniform continuity claim 1}
            d_{Y}\left(h_n^{x}(y),h_n^{x'}(y)\right)<\epsilon,
        \end{equation}
        for all $y\in B_L(y_0)$ and $n=1,2,\dots, N_0$. Choose $N_1$ such that 

        \begin{equation}\label{Choice of N_1 in claim 1}
            \mu^{N_1}M<\delta.
        \end{equation}
        Now, suppose $m,n>\tilde N:=N_0+N_1$ with $m>n$. For $j=1,\dots, N_0$, define
        $$A_{j}:=d_{Y}\left(\pi_Y\circ F_j\circ\cdots\circ F_m(x_0,y_0), h_j^{f_{j+1}\cdots f_m(x_0)}\left( \pi_Y\circ F_{j+1}\circ\cdots \circ F_n(x_0,y_0)\right)\right),$$
        $$B_j:=d_Y\left(h_j^{f_{j+1}\cdots f_m(x_0)}\circ \pi_Y \circ F_{j+1}\circ\cdots \circ F_n(x_0,y_0), \pi_Y \circ F_j\circ \cdots \circ F_n(x_0,y_0)\right) ,$$
        and
        $$C_j:=d_Y\left( \pi_Y\circ F_j\circ\cdots\circ F_m(x_0,y_0), \pi_Y\circ F_j\circ \cdots \circ F_n(x_0,y_0)\right).$$
We will show that
$$C_1\leq \left(\text{Constant}\right) \cdot \epsilon $$
for some constant that only depends on $\lambda$. First, we derive relations between the quantities $A_j, \ B_j$ and $C_j$. 

    Notice that from (\ref{projected coordinate}) and (\ref{Upper bound for iterations in Y}), we have
    \begin{equation}\label{Upper bound for C claim 1}
        C_{N_0}\leq 2L
    \end{equation}
and by the triangle inequality, we have
        \begin{equation}\label{C compared to A and B}
            C_j\leq A_j+B_j.
        \end{equation}
        In addition, since 
\begin{align*}
    A_j=d_{Y}\left(h_j^{f_{j+1}\cdots f_m(x_0)}\left(\pi_Y\circ F_{j+1}\circ\cdots\circ F_m(x_0,y_0)\right), h_j^{f_{j+1}\cdots f_m(x_0)}\left( \pi_Y\circ F_{j+1}\circ\cdots \circ F_n(x_0,y_0)\right)\right),
\end{align*}
     by (\ref{fiber contraction rate}), this implies that
        \begin{equation}\label{A compared to C}
            A_j\leq \lambda C_{j+1}.
        \end{equation}
        Also, since $j\leq N_0$, we have that 
        $$n-{j}\geq \tilde N-{N_0}=N_1$$
        so from (\ref{base contraction rate}), (\ref{base iteration bound}), and (\ref{Choice of N_1 in claim 1}), we have 
        $$d_X\left(f_{j+1}\circ\cdots \circ f_m(x_0),f_{j+1}\circ \cdots \circ f_{n}(x_0)\right)\leq \mu^{n-{j}}d_X(f_{n+1}\circ\cdots\circ f_m(x_0),x_0) \leq \mu^{N_1}M<\delta.$$
        Since
        $$B_j=d_Y\left(h_j^{f_{j+1}\cdots f_m(x_0)}\left(\pi_Y \circ F_{j+1}\circ\cdots \circ F_n(x_0,y_0)\right), h_j^{f_{j+1}\cdots f_n(x_0)} \circ\left(\pi_Y \circ F_{j+1}\circ \cdots \circ F_n(x_0,y_0)\right)\right),$$
        in combination with (\ref{Upper bound for iterations in Y}) and (\ref{Uniform continuity claim 1}), this implies that
        \begin{equation}\label{B less than epsilon}
            B_j<\epsilon,
        \end{equation}
for all $j\leq N_0$. By a repeated application of equations (\ref{C compared to A and B}), (\ref{A compared to C}), and (\ref{B less than epsilon}), we have 
        \begin{align*}
            &d_Y\left(\pi_Y\circ F_1\circ\cdots\circ F_m(x_0,y_0), \pi_Y\circ F_1\circ\cdots\circ F_n(x_0,y_0)\right)=C_1\\
            &\leq A_1+B_1\leq \lambda C_2+\epsilon\\
            &\leq \lambda(A_2+B_2)+\epsilon\leq \lambda^2C_3+\lambda\epsilon+\epsilon\\
            &\leq \lambda^2(A_3+B_3)+\lambda\epsilon+\epsilon\leq \lambda^3C_4+\lambda^2\epsilon+\lambda\epsilon+\epsilon\\
            &\leq \dots\leq \lambda^{N_0-1}C_{N_0}+\lambda^{n-2}\epsilon+\cdots+\lambda\epsilon+\epsilon\\
            &\leq 2\lambda^{N_0-1}L+\lambda^{n-2}\epsilon+\cdots+\lambda\epsilon+\epsilon\\
            &\leq \epsilon+\epsilon \sum_{i=0}^{n-2}\lambda^i\leq \left(1+\sum_{i=0}^{\infty}\lambda^i\right)\cdot \epsilon,
        \end{align*}
        where the third to last inequality follows from (\ref{Upper bound for C claim 1}) and the second to last follows from (\ref{choice of N_0 claim 1}). We conclude that $\left(\pi_Y\circ F_1\circ\cdots \circ F_n(x_0,y_0)\right)$ is a Cauchy sequence in $Y$ and hence convergent, since $Y$ is complete.
    \end{proof}
    Denoting the limit of the sequence $\left(\pi_Y\circ F_1\circ\cdots \circ F_n(x_0,y_0)\right)_{n\in\mathbb N}$ by $y^*$, we see that for any $y\in Y$, we have 
    $$d_Y\left(\pi_Y\circ F_1\circ\cdots \circ F_n(x_0,y), \pi_Y\circ F_1\circ\cdots \circ F_n(x_0,y_0)\right)\leq \lambda^nd_Y(y,y_0)\rightarrow 0$$
    as $n\rightarrow \infty$, so that
    \begin{equation}\label{independence of y_0}
        \lim\limits_{n\rightarrow\infty} \pi_Y\circ F_1\circ\cdots \circ F_n(x_0,y)=y^*.
    \end{equation}
    \begin{claim}\label{limit independent of x}
    For each $x\in X$ and $y\in Y$, we have 
$$\lim\limits_{n\rightarrow\infty}\pi_Y\circ F_1\circ\cdots\circ F_n(x,y)=y^*.$$
\end{claim}
\begin{proof}[Proof of Claim \ref{limit independent of x}]
Let $\epsilon > 0$, $x \in X$, $y \in Y$. Set $L':=d_Y(y,y_0)+L$. By property (2), there is a $S'>0$ such that if $d_X\left(x',x''\right)<d_X(x,x_0)$, then 
\begin{equation}\label{bound S' in claim 2}
   d\left(h_n^{x'}(y), h_n^{x''}(y)\right)\leq  S'
\end{equation}
for all $y\in B_{L'}(y_0)$ and $n\in\mathbb N$.
Choose $N_0$ so large such that 
\begin{equation}\label{choice of N_0 claim 2}
    S' \sum^{\infty}_{i =N_0} \lambda^{i} < \epsilon.
\end{equation}
By condition (3), there is a $\delta>0$ such that if $d_X(x',x'')<\delta$, then 
        \begin{equation}\label{Uniform continuity claim 2}
            d_{Y}\left(h_n^{x'}(y),h_n^{x''}(y)\right)<\epsilon,
        \end{equation}
    for all $y\in B_{L'}(y_0)$  and $n=1,2,\dots, N_0$.
Choose $N_1$ so large such that 
\begin{equation}\label{Choice of N_1 claim 2}
    \mu^{N_1}d(x,x_0)<\delta.
\end{equation}
Suppose $n > \tilde N:=N_0 + N_1$. Let us adopt the notation 
$$T_k^x:=h_1^{f_2\cdots f_n x}\circ\cdots \circ h_k^{f_{k+1}\cdots f_n x}$$
so that 
$$\pi_Y\circ F_1\circ\cdots \circ F_n(x,y)=T_{k-1}^{x}\left(\pi_Y\circ F_k\circ\cdots \circ F_n(x,y)\right)$$
for each $k\leq n$. Define
$$A_k:=d_Y\left(T_{k-1}^x(\pi_Y\circ F_{k}\circ \cdots \circ F_n(x,y), T_{k-1}^x\left(\pi_Y\circ F_{k}\circ \cdots \circ F_n(x_0,y) \right) \right)$$
$$B_k:=d_Y\left(h_k^{f_{k+1}\cdots f_nx}\left(\pi_Y\circ F_{k+1}\circ\cdots \circ F_n(x_0,y)\right), h_k^{f_{k+1}\cdots f_nx_0} \left(\pi_Y\circ F_{k+1}\circ\cdots \circ F_n(x_0,y)\right)\right)$$
and
$$C_k := d_Y\left(T_{k}^x\left(\pi_Y\circ F_{k+1}\circ \cdots \circ F_n(x_0,y) \right),T^{x_0}_n(y_0) \right)$$
for $k\leq n$.

We will show that 
$$C_n\leq (\text{Constant}) \cdot \epsilon$$
for some constant that only depends on $\lambda$, but we first establish relations between the quantities $A_k$, $B_k$ and $C_k$.

First notice that by definition of $A_k$ and $B_k$, and (\ref{fiber contraction rate}), we have
\begin{equation}\label{relationship A and B Claim 2}
    A_k \leq \lambda^{k-1}B_k,
\end{equation}
and by the triangle inequality, 
\begin{equation}\label{relationship A and C claim 2}
    C_k \leq A_k + C_{k-1}.
\end{equation}
From (\ref{Upper bound for iterations in Y}), (\ref{fiber contraction rate}) and the triangle inequality, we have
\begin{align*}
 &d_Y(\pi_Y\circ F_{k+1}\circ\cdots \circ F_n(x_0,y),y_0)\\
 &\leq  d_Y(\pi_Y\circ F_{k+1}\circ\cdots \circ F_n(x_0,y),\pi_Y\circ F_{k+1}\circ\cdots \circ F_n(x_0,y_0))\\
 & \hspace{5 cm} +d_Y(\pi_Y\circ F_{k+1}\circ\cdots \circ F_n(x_0,y_0),y_0)\\
 &\leq \lambda^{n-k}d_Y(y,y_0)+d_Y(\pi_Y\circ F_{k+1}\circ\cdots \circ F_n(x_0,y_0),y_0)\\
 &\leq  d_Y(x_0,y_0)+L=L'.
\end{align*}
and also 
$$d_X\left( f_{k+1}\circ\cdots \circ f_n(x), f_{k+1}\circ\cdots \circ f_n(x_0)\right)\leq \mu^{n-k}d(x,x_0)\leq d_X(x,x_0)$$
for all $k\leq n$. 
So, from (\ref{bound S' in claim 2}), we have
\begin{equation}\label{Bound for B claim 2}
    B_k\leq S',
\end{equation}
for all $k\leq n$. Also, for $k \leq N_0$, we have that
        $$n-{k}\geq \tilde N-{N_0}=N_1,$$
        so from (\ref{base contraction rate}) and (\ref{Choice of N_1 claim 2}), we know that 
        $$d_X\left(f_{k+1}\circ\cdots \circ f_n(x_0),f_{k+1}\circ \cdots \circ f_{n}(x)\right)\leq \mu^{n-{k}}d_X(x,x_0)\leq \mu^{N_1}d(x,x_0)<\delta.$$
        Thus, from (\ref{Uniform continuity claim 2}), we have
\begin{align}\label{B less than epsilon claim 2}
B_k \leq \epsilon
\end{align}
for all $k\leq N_0$. Now, using a repeated application of (\ref{relationship A and B Claim 2}) and (\ref{relationship A and C claim 2}), we have
\begin{align*}
    &d_Y\left(\pi_Y\left(F_1\circ\cdots \circ F_n(x,y)\right), \pi_Y(F_1\circ\cdots \circ F_n(x_0,y_0))\right)= C_n\\
    &\leq A_n+C_{n-1}\leq \lambda^{n-1}B_n+A_{n-1} + C_{n-2}  \\
    &\leq \lambda^{n-1}B_n + \lambda^{n-2}B_{n-1} + A_{n-2} + C_{n-3}\\
    &\leq \lambda^{n-1}B_n + \lambda^{n-2}B_{n-1} + \lambda^{n-3}B_{n-2} + A_{n-3}+C_{n-4}\\
    &\leq \dots\leq \lambda^{n-1}B_n+\lambda^{n-2}B_{n-1}+\cdots+\lambda B_2+B_1
\end{align*}
and from (\ref{Bound for B claim 2}), (\ref{B less than epsilon claim 2}) and (\ref{choice of N_0 claim 2}), the last quantity satisfies 
\begin{align*}
    &\sum_{k=1}^{n}\lambda^{k-1}B_k \leq \sum^{N_0}_{k=1}\lambda^{k-1}B_k + \sum^{n}_{k=N_0+1}\lambda^{k-1}B_k \\
    &\leq \epsilon \sum^{N_0}_{k=1}\lambda^{k-1} + \epsilon\leq \left(1+\sum_{k=1}^{\infty}\lambda^{k-1}\right)\cdot \epsilon.
\end{align*}
Therefore, 
$$ \lim_{n \rightarrow\infty}d_Y\left(\pi_Y\left(F_1\circ\cdots \circ F_n(x,y)\right), \pi_Y(F_1\circ\cdots \circ F_n(x_0,y_0))\right) = 0,$$
so from (\ref{independence of y_0}), the claim holds.
\end{proof}
Combining these claims with Lemma \ref{base lemma} implies that there is some $x^*\in X$ such that 
$$\lim\limits_{n\rightarrow\infty}F_1\circ \cdots \circ F_n(x,y)=\left(x^*,y^*\right)$$
for all $(x,y)\in X\times Y$.
\end{proof}

\begin{remark}
   Notice that condition (4) can be reformulated as the assumption that for bounded $K\subset Y$, the family of maps $\left\{x\mapsto h_n^{x}\right\}_{y\in K}$ is uniformly equicontinuous for each $n\in\mathbb N$. This is analogous to the continuity assumption in the Fiber Contraction Theorem.
\end{remark}

We now give examples demonstrating the conditions (2) and (3) cannot be removed in Theorem \ref{main theorem}. Notice that condition (1) cannot be removed by Remark \ref{remark on boundedness assumption in base lemma}.

\begin{example}
    Condition (2) cannot be removed for if we set $X=Y=\mathbb R$, and for all $x,y \in \mathbb{R}$ and $n \in \mathbb{N}$, set $f_n(x)=\frac{1}{2}x$ and $h_n^x(y)=\frac{1}{2}y+3^n$, then 
    $$\lim\limits \pi_Y\circ F_1\circ\cdots\circ F_n(0,0)=\sum_{i=1}^n \frac{3^i}{2^{i-1}}\rightarrow \infty$$
    as $n\rightarrow \infty$.
\end{example}

\begin{example}
    Condition (3) cannot be removed for if we set $X=Y=\mathbb R$, and for all $x,y\in \mathbb R$ and $n\in\mathbb N$, we have $f_n(x)=\frac{1}{2}{x}$ and 
    $$h_n^x(y)=\begin{cases} 
      0 & x=0 \\
      \frac{1}{2}\left(y-\frac{1}{4}\right)+\frac{1}{4} & x\neq 0
   \end{cases},$$
 then 
    $$\lim\limits_{n\rightarrow\infty} F_1\circ\cdots\circ F_n(x,y)=\begin{cases} 
      (0,0) & \text{if} \ x= 0 \\
      \left(0,\frac{1}{4}\right) & \text{if} \ x\neq 0
   \end{cases}.$$
\end{example}

 \section*{Acknowledgements}
We would like to thank Anton Gorodetski for checking the validity of our statements and offering suggestions to improve the quality of the text. The first author was supported by NSF grant DMS-2247966 (PI: A. Gorodetski).

    \newcommand{\Addresses}{{
  \bigskip
  \footnotesize

  \textsc{Department of Mathematics, University of California, Irvine, CA 92697, USA}\par\nopagebreak
  \textit{E-mail address}: \texttt{lunaar1@uci.edu}\\
  \\

\textsc{Department of Mathematics, University of California, Irvine, CA 92697, USA}\par\nopagebreak
  \textit{E-mail address}: \texttt{weirany1@uci.edu}
}}

\Addresses

\footnotesize

\begin{thebibliography}{99}
\bibitem[Ar]{Ar} Arnold L.,  Random dynamical systems, {\it Springer Monographs in Mathematics}, Springer-Verlag, Berlin, 1998. xvi+586 pp.

\bibitem[BIST]{bist} Bellissard J.,  Iochum B., Scoppola E., and  Testard D., Spectral properties of one-dimensional
quasicrystals, {\it Commun. Math. Phys.}, vol. 125 (1989), pp. 527--543.

\bibitem[BS]{BS} Brooks R.M., Schmitt K. The contraction mapping principle and some applications. (2009). Electronic Journal of Differential Equations, 1(Mon. 01-09), Mon. 09, 1-90. https://doi.org/10.58997/ejde.mon.09

\bibitem[C]{C} Casdagli M., Symbolic dynamics for the renormalization map
of a quasiperiodic Schr\"odinger equation, \textit{Comm.\ Math.\
Phys.}\ \textbf{107} (1986), 295--318.


\bibitem[Ca]{Ca} Cantat S., Bers and H\'enon, Painlev\'e and Schr\"odinger, \textit{Duke Math.\ J.}\ \textbf{149} (2009), 411--460.
\bibitem[CRV]{CRV}  Castro A., Rodrigues F.B., Varandas P.,
Stability and limit theorems for sequences of uniformly hyperbolic dynamics,
\textit{Journal of Mathematical Analysis and Applications}, Volume 480, Issue 2, 2019.

\bibitem[D]{D2000} Damanik D., Substitution Hamiltonians with bounded trace map orbits, {\it J. Math. Anal. Appl.}, vol. 249 (2000), no. 2, pp. 393--411.

\bibitem[DEGT]{degt} Damanik D., Embree M., Gorodetski A., and Tcheremchantsev S.,
The Fractal Dimension of the Spectrum of the Fibonacci Hamiltonian, {\it Communications in Mathematical Physics}, vol. 280 (2008), no. 2, pp. 499-516.


\bibitem[DG1]{DG1} Damanik D., Gorodetski A., Hyperbolicity of the Trace Map for the Weakly Coupled Fibonacci Hamiltonian,  {\it Nonlinearity,}  vol. 22 (2009), pp. 123--143.


\bibitem[DG2]{dg3} Damanik D., Gorodetski A., Spectral and Quantum Dynamical Properties of the Weakly Coupled Fibonacci Hamiltonian, {\it  Communications in Mathematical Physics}, vol. 305 (2011), pp. 221--277.


\bibitem[DGY]{DGY} Damanik D., Gorodetski A., Yessen W.N., The Fibonacci Hamiltonian. \textit{Inventiones mathematicae}, 206 (2014), 629-692.



\bibitem[GJK]{GJK} Gorodetski A., Jang S.u., Kleptsyn V., Sturmian Trace Maps, in preparation.

\bibitem[HP]{hp} Hirsch M., Pugh C., Stable manifolds and hyperbolic sets. 1970 {\it Global Analysis (Proc. Sympos. Pure Math., Vol. XIV, Berkeley, Calif., 1968)} pp. 133--163 {\it Amer. Math. Soc., Providence, R.I.}

\bibitem[HPS]{HPS} Hirsch M., Pugh C., Shub M., Invariant Manifolds, Lecture Notes in Math. \textbf{583},
Springer-Verlag, Berlin, 1977.

\bibitem[J]{J} Jachymski J., Continuous dependence of fixed points on parameters,
              remetrization theorems, and an application to the initial
              value problem. \textit{Math. Methods Appl. Sci.} 46 (2023).
\bibitem[JJT]{JJT} Jachymski J., Jóźwik I., Terepeta M. The Banach Fixed Point Theorem: selected topics from its hundred-year history. \textit{Rev. Real Acad. Cienc. Exactas Fis. Nat. Ser. A-Mat.} 118, 140 (2024). https://doi.org/10.1007/s13398-024-01636-6



\bibitem[KH]{KH}  Katok A., Hasselblat B., Introduction to the Modern Theory of Dynamical System, Cambridge Univ. Press, 1995.

\bibitem[L1]{L} Luna A., On the spectrum of Sturmian Hamiltonians in a small coupling regime, (2024). arXiv:2408.01637

    \bibitem[L2]{Lu2} Luna A., Regularity of Non-stationary Stable Manifolds of Toral Anosov Maps, Pre-print (2024). 

\bibitem[M]{M} M.\,Mei, Spectral properties of discrete Schr\"odinger operators with potentials generated by primitive invertible substitutions, {\it J. Math. Phys}. 55, 082701 (2014).

\bibitem[Mu1]{Mu1} Muentes Acevedo J. de J., Openness of Anosov Families. Journal of the Korean Mathematical Society, 55(3) (2018), 575–591. https://doi.org/10.4134/JKMS.J170312

\bibitem[Mu2]{Mu2} Muentes Acevedo J. de J., Structural stability and a characterization of Anosov families. \textit{Dynamical Systems}, 34(3) (2018), 399–421. https://doi.org/10.1080/14689367.2018.1546380

    \bibitem[Mu3]{Mu} Muentes Acevedo J. de J., Local stable and unstable manifolds for Anosov families, \textit{Hokkaido Math. J.}, 48 (2019), 513--535. 

   
\bibitem[S]{S} Stenlund M., Non-stationary compositions of Anosov diffeomorphisms, \textit{Nonlinearity}, \textbf{24}, No.10 (2011)



\bibitem[ZLZ]{ZLZ} Zhang W., Lu K., Zhang W., Smooth invariant foliations without a bunching condition and Belitskii's $C^1$ linearization for random dynamical systems, (2023). arXiv:2307.11284
\end{thebibliography}
\end{document}